\documentclass[11pt,leqno,a4paper]{amsart}
\usepackage{amsmath}
\usepackage{amsthm}
\usepackage{amssymb}
\usepackage{amscd}
\usepackage{mathabx}
\usepackage{accents}

\usepackage{tikz, stmaryrd}
\usetikzlibrary{matrix,arrows,decorations.pathmorphing}

\usepackage{enumerate}
\usepackage{verbatim}

\usepackage{mathabx}

\usepackage[english]{babel}
\usepackage[utf8]{inputenc}
\usepackage[totalheight=22 true cm, totalwidth=14 true cm]{geometry}
\usepackage{mathrsfs}
\usepackage[all]{xy}

\newtheorem{thm}{Theorem}[section]
\newtheorem{cor}[thm]{Corollary}
\newtheorem{lemma}[thm]{Lemma}

\newtheorem{prop}[thm]{Proposition}

\newtheorem{defn}[thm]{Definition}

\theoremstyle{remark}

\theoremstyle{definition}

\newtheorem{rmk}[thm]{Remark}
\newtheorem{exa}[thm]{Example}

\newtheorem{notation}[thm]{Notation}
\numberwithin{equation}{thm}

\def\beq{\begin{equation}}
\def\eeq{\end{equation}}

\def\beqn{\begin{equation*}}
\def\eeqn{\end{equation*}}

\def\ben{\begin{enumerate}}
\def\een{\end{enumerate}}

\def\crash#1{}

\def\C{{\mathbb C}}

\def\N{{\mathbb N}}

\def\R{{\mathbb R}}

\def\l{\left}
\def\r{\right}

\def\Ker{{\rm Ker \,}}

\def\ie{\emph{i.e.}~}

\def\cA{{\mathcal A}}
\def\cB{{\mathcal B}}

\def\cE{{\mathcal E}}

\def\cI{{\mathcal I}}

\def\cN{{\mathcal N}}

\def\cV{{\mathcal V}}

\def\cW{{\mathcal W}}

\def\sP{{\mathscr P}}

\def\wtilde{\widetilde}

\def\a{\alpha}
\def\be{\beta}

\def\la{\lambda}

\def\ol{\overline}

\def\limpro{\mathop{\lim\limits_{\displaystyle\leftarrow}}}

\def\limind{\mathop{\lim\limits_{\displaystyle\rightarrow}}}

\def\then{\Rightarrow}

\def\lt{\langle}
\def\gt{\rangle}

\def\void{{\rm \varnothing}}

\def\bBorn{\mathbf{Born}}
\def\bLoc{\mathbf{Loc}}

\def\uborn{{\rm uborn\,}}
\def\Ban{{\rm Ban\,}}

\author{Federico Bambozzi}

\title{Closed graph theorems for bornological spaces}

\thanks{
        The author acknowledges the University of Regensburg with the support of the DFG funded CRC 1085 "Higher Invariants. Interactions between Arithmetic Geometry and Global Analysis". }

\begin{document}

\begin{abstract}
	The aim of this paper is that of discussing Closed Graph Theorems for bornological vector spaces in a way which is accessible to non-experts. We will see how to easily adapt classical arguments of functional analysis over $\R$ and $\C$ to deduce Closed Graph Theorems for bornological vector spaces over any complete, non-trivially valued field, hence encompassing the non-Archimedean case too. We will end this survey by discussing some applications. In particular, we will prove de Wilde's Theorem for non-Archimedean locally convex spaces and then deduce some results about the automatic boundedness of algebra morphisms for a class of bornological algebras of interest in analytic geometry, both Archimedean (complex analytic geometry) and non-Archimedean.
\end{abstract}

\maketitle


\section*{Introduction}

This paper aims to discuss the Closed Graph Theorems for bornological vector spaces in a way which is accessible to non-specialists and to fill a gap in literature about the non-Archimedean side of the theory at the same time. In functional analysis over $\R$ or $\C$ bornological vector spaces have been used since long time ago, without becoming a mainstream tool. It is probably for this reason that bornological vector spaces over non-Archimedean valued fields were rarely considered. Over the last years, for several reasons, bornological vector spaces have drawn new attentions: see for example \cite{Bam}, \cite{BaBe}, \cite{BaBeKr}, \cite{CAM} and \cite{M}. These new developments involve the non-Archimedean side of the theory too and that is why it needs adequate foundations. Among the classical books on the theory of bornological vector spaces, the only one considering also non-Archimedean base fields, in a unified fashion with the Archimedean case, is \cite{H2}. But that book cannot cover all the theory, and in particular it lacks of a unified treatment of the Closed Graph Theorems. This work comes out from the author's need for a reference for these theorems and also in the hope that in the future bornological vector spaces will gain more popularity and that this work may be useful for others. 

The Closed Graph Theorem for Banach spaces over $\R$ and $\C$ is one of the most celebrated classical theorems of functional analysis. Over the years it has been generalized in many ways, for example to Fr\'echet and LF spaces as a consequence of the Open Mapping Theorems. Although it is a classical argument that the Closed Graph Theorem can be deduced from the Open Mapping Theorem, people have understood that the Closed Graph Theorem can be proved in an independent way, with argumentations that also work when the Open Mapping Theorem fails. The two most famous examples of this kind of proofs are given in \cite{POPA} by Popa and \cite{DW} by de Wilde. In particular, de Wilde's Theorem is probably the most general Closed Graph Theorem for locally convex spaces, and states the following:

\begin{thm} (De Wilde's Closed Graph Theorem)\\
	If $E$ is an ultrabornological locally convex space and
	$F$ is a webbed locally convex space over $\R$ or $\C$, then every linear map $f: E \to F$ which has bornologically closed graph with respect to the convex bornologies on $E$ and $F$ that are generated by all bounded Banach disks in $E$ and in $F$, respectively, is continuous even if regarded as a mapping into the ultrabornologification $F_\uborn$ of $F$.
\end{thm}

The terminology of the theorem will be explained in the course of this work. What is interesting to notice is that, although we would like to prove a theorem for locally convex spaces we are naturally led to talk about bornologies and bounded maps. Popa's Theorem, on the contrary, is an explicit bornological statement which is the Archimedean case of our Theorem \ref{thm:net}. 

The content of the paper is the following: in the first section we give an overview of the theory of bornological vector spaces. In particular, since we adopt the unusual attitude of discussing the Archimedean and non-Archimedean case of the theory at the same time, we spend some time in recalling basic definitions and discuss in details the notions from the theory bornological vector spaces that will be used. In the second section we will introduce the notion of bornological nets and then give the main examples of bornological vector spaces endowed with nets. We will then deduce our first Closed Graph Theorem, which is the unified version of Popa's Theorem (cf. \cite{POPA}), stated as follows: 

\begin{thm} 
	Let $E$ and $F$ be separated convex bornological vector spaces, where $E$ is complete and $F$ has a net compatible with its bornology. Then, every linear map $f: E \to F$ with bornologically closed graph is bounded.
\end{thm}

In the subsequent section the notion of bornological net is generalized by the notion of bornological web and the analogous Closed Graph Theorem for webbed bornological vector spaces is proved quite easily as a consequence of the previous discussion. In this case our theorem is the direct generalization, for all base fields, of the Closed Graph Theorem proved by Gach in \cite{G}, Theorem 4.3. 
In the last section we discuss some applications of the theorems we proved. We will see how one can deduce Isomorphism Theorems from Closed Graph Theorems and following \cite{G} we will see how de Wilde's Closed Graph Theorem can be deduced. We would like to remark that for non-Archimedean base fields we need to add some restrictions, that do not affect the Archimedean side of the theory. Our generalization of de Wilde's Theorem is the following:

\begin{thm}
	If $E$ is an ultrabornological locally convex space and
	$F$ is a polar webbed locally convex space defined over a spherically complete field $K$, then every linear map $f: E \to F$ which has bornologically closed graph with respect to the convex bornologies on $E$ and $F$ that are generated by all bounded Banach disks in $E$ and in $F$, respectively, is continuous even if regarded as a mapping into the ultrabornologification $F_\uborn$ of $F$.
\end{thm}

Therefore, in order to deduce de Wilde's Theorem for non-Archimedean base fields we needed to suppose that the base field $K$ is spherically complete and that $F$ is a polar locally convex space (cf. Definition \ref{defn:polar}), conditions which are always satisfied when $K$ is Archimedean. We remark that these hypothesis on $K$ and $F$ are only used in Lemma \ref{lemma:polar} and in Lemma \ref{lemma:web_ultraborn}; one might ask if it is possible to prove that lemmata without these restrictions. We do not address this problem in this work.
Finally, in the last part of the paper we show how to use the Closed Graph Theorems for bornological spaces to deduce that all algebra morphisms between dagger affinoid algebras, as defined in \cite{Bam}, are bounded. This application, and others coming in \cite{BaBeKr} and planned in future works, are our main motivations for this study. 

\section{Bornological spaces and Closed Graph Theorems}

\subsection*{Closed graph}
Let $u: E \to F$ be a map of sets, then the set 
\[ \Gamma(u) =  \{ (x, y) \in E \times F \ | y = u(x) \ \}  \]
is called the graph of $u$. If $E$ and $F$ are Hausdorff topological spaces and $u$ a continuous map, then $\Gamma(u)$ is a closed subspace of $E \times F$ endowed with the product topology. This is a basic property of Hausdorff topological spaces. If $E$ and $F$ are vector spaces over a field $K$ and $u$ is a linear map, then $\Gamma(u)$ is a vector subspace of $E \times F$. If $E$ and $F$ are separated bornological vector spaces over a complete, non-trivially valued field $K$ and $u$ is linear and bounded, then $\Gamma(u)$ is bornologically closed in $E \times F$ endowed with the product bornology (see below for what it means for a subset of a bornological vector space to be bornologically closed). This  assertion is pretty easy to check. Let $(x_n, u(x_n))$ be a sequence of elements of $\Gamma(u)$ which converges bornologically to $( x, y)$ in $E \times F$. Then, by definition of product bornology, $x_n \to x$ in $E$ and $u(x_n) \to y$ in $Y$. Since $u$ is bounded, the sequence $(u(x_n))$ converges bornologically to $u(x)$ and since $Y$ is separated, we must have $y = u(x)$. Therefore $(x, y) \in \Gamma(u)$ and $\Gamma(u)$ is bornologically closed in $E \times F$.

The Closed Graph Theorems are converses of the above statements for some special class of bornological or topological $K$-vector spaces. Here we pursue the main ideas of \cite{G} for which the bornological Closed Graph Theorems are of more fundamental nature and extend it for non-Archimedean base fields.
The statements of our Closed Graph Theorems assert that if $u: E \to F$ is a linear map which has bornologically closed graph then it is bounded when $E$ and $F$ belong to some particular classes of bornological vector spaces: we will prove it when $E$ is a complete bornological vector space and $F$ a separated bornological vector space endowed with a net, in Section \ref{sec:closed_nets}, and when $F$ is endowed with a web in Section \ref{sec:closed_webs}. Both the proofs of these sections are adaptations of results from \cite{H} and \cite{G} on bornological vector spaces over Archimedean base fields to any non-trivially valued, complete base field $K$ Archimedean or not.
 
\subsection*{Bornologies}

Bornological vector spaces are well studied objects in functional analysis over $\R$ and $\C$. They are not mainstream, as the theory of topological or locally convex vector spaces, but they are often useful in addressing problems for locally convex spaces and they found a good amount of applications. Thus, during the years, a good amount of work has been done to study the properties of bornological vector spaces and algebras over $\R$ and $\C$: for examples \cite{H}, \cite{Hog2}, \cite{Hog3}, \cite{WAEL}, \cite{M} discuss various aspects and applications of the theory. On the other hand, for non-Archimedean base fields the theory has never got much attention. The only works known to the author which deal with bornological vector spaces over non-Archimedean base fields date back to many years ago and they often discuss the non-Archimedean base field case as a mere example, interesting for working out general theories and general theorems, pursuing a ``Bourbaki" study of the subject; however, this side of the theory were seldom thought to have applications to ``real" mathematics. An example of this kind of work is \cite{H2}. In recent years a renewed interest in the theory of bornological vector spaces may challenge this attitude: for example, on the Archimedean side of the theory, one has the work \cite{M}, \cite{M2}, \cite{M3} showing the usefulness of bornological vector spaces in non-commutative geometry and representation theory and \cite{Bam}, \cite{CAM}, \cite{BaBe} where also the theory of non-Archimedean bornological vector spaces and algebras is used in analytic geometry and higher order local field theory. We will recall the basis of the theory of bornoloogical vector spaces as far as needed for our scopes.

\begin{defn}
Let $X$ be a set. A \emph{bornology} on $X$ is a collection $\cB$ of subsets of $X$ such that
\begin{enumerate}
\item $\cB$ is a covering of $X$, \ie $\forall x \in X, \exists B \in \cB$ such that $x \in \cB$;
\item $\cB$ is stable under inclusions, \ie $A \subset B \in \cB \then A \in \cB$;
\item $\cB$ is stable under finite unions, \ie for each $n \in \N$ and $B_1, ..., B_n \in \cB$, $\underset{i = 1}{\overset{n}\bigcup} B_i \in \cB$.
\end{enumerate}

A pair $(X, \cB)$ is called a \emph{bornological set}, and the elements of $\cB$ are called \emph{bounded subsets} of $X$ 
(with respect to $\cB$, if it is needed to specify).  
A family of subsets $\cA \subset \cB$ is called a \emph{basis} for $\cB$ if for any $B \in \cB$ there exist $A_1, \dots, A_n \in \cA$ such that $B \subset A_1 \cup \dots \cup A_n$.
A \emph{morphism} of bornological sets $\varphi: (X, \cB_X) \to (Y, \cB_Y)$ is defined to be a bounded map $\varphi: X \to Y$, \ie a map of sets such that $\varphi(B) \in \cB_Y$ for all $B \in \cB_X$. 
\end{defn}

From now on let's fix a complete, non-trivially valued field $K$. We will use the notation $K^\circ = \{ x \in K | |x| \le 1 \}$. Clearly $K$ has a natural structure of bornological field.

\begin{defn}
A \emph{bornological vector space} over $K$ is a $K$-vector space $E$ along with a bornology on the underlying set of $E$ for which the maps $(\la, x) \mapsto \la x$ and $(x, y) \mapsto x + y$
are bounded.
\end{defn}

We will only work with bornological vector spaces whose bounded subsets can be described using convex subsets, in the following way.

\begin{defn}
Let $E$ be a $K$-vector space. A subset $B \subset E$ is called \emph{balanced} if for every $\la \in K^\circ$ one has that $\la B \subset B$. A subset $B \subset E$ is called \emph{absolutely convex} (or \emph{disk}) if
\begin{enumerate}
\item for $K$ Archimedean, it is convex and balanced, where convex means that for every $x, y \in B$ and $t \in [0, 1]$ then $(1 -t) x + t y \in B$;
\item for $K$ non-Archimedean, it is a $K^\circ$-submodule of $E$.
\end{enumerate}
\end{defn}

The definition of absolutely convex subset of $E$ is posed in two different ways, depending on $K$ being Archimedean of non-Archimedean, although the formal properties are essentially the same in both cases. Moreover, using the theory of generalized rings (in the sense of Durov \cite{DUR}) one can put the two situations on equal footing, but we are not interested to this issue in this work. 
 
\begin{defn}
A bornological $K$-vector space is said to be \emph{of convex type} if its bornology has a basis made of absolutely convex subsets. 
\end{defn}

We denote by $\bBorn_K$ the category of bornological vector spaces of convex type over $K$. Then, every semi-normed space over $K$ can be endowed with the bornology induced by the bounded subsets for the semi-norm and is manifestly of convex type. This association permits to see the category of semi-normed spaces over $K$ as a full sub-category of $\bBorn_K$.

\begin{rmk}
For every bornological vector space of convex type $E$ there is an isomorphism
\[ E \cong \limind_{B \in \cB_E} E_B  \]
where $\cB_E$ denotes the family of bounded absolutely convex subsets of $E$ and $E_B$ is the vector subspace of $E$ spanned by the elements of $B$  equipped with the \emph{gauge semi-norm} (also called \emph{Minkowski functional}) defined by $B$. Notice that all morphism of this inductive system are monomorphisms.
\end{rmk}

\begin{defn} \label{defn:separated_born}
        A bornological vector space over $K$ is said to be \emph{separated} if its only bounded vector subspace is the trivial subspace $\{0\}$.
\end{defn}

\begin{rmk}
A bornological vector space of convex type over $K$ is separated if and only if for each $B\in \cB_E$, the gauge semi-norm on $E_{B}$ is actually a norm.
\end{rmk}

\begin{defn} \label{defn:banach_disk}
	A disk $B \subset E$ is said to be a \emph{Banach disk} if $E_B$ is a Banach space.
\end{defn}

\begin{defn} \label{defn:complete_born}
        A bornological space $E$ over $k$ is said to be \emph{complete} if
        \[ E \cong \limind_{i \in I} E_i \]
        for a filtered colimit of Banach spaces over $K$ for which the system morphisms are all injective.
\end{defn}

It can be shown that the definition of complete bornological vector space just given is equivalent to the request that the family of bounded disks $\cB_E$ of $E$ admits a final subfamily made of Banach disks, cf. \cite{H2} Proposition 7 page 96.

\subsection*{Bornological convergence and bornologically closed subsets} The data of a bornology is enough for endowing a vector space of a notion of convergence.

\begin{defn} \label{defn:born_conv}
	Let $E$ be a bornological $k$-vector space and $\{x_n\}_{n \in \N}$ a sequence of elements of $E$. We say that $\{x_n\}_{n \in \N}$ \emph{converges (bornologically) to $0$ in the sense of Mackey} if there exists a bounded subset $B \subset E$ such that for every $\la \in K^\times$ there exists an $n = n(\la)$ for which
	\[ x_m \in \la B, \forall m > n. \] 
	We say that $\{ x_n \}_{n \in \N}$ converges  (bornologically) to $a \in E$ if the sequence $\{x_n - a\}_{n \in \N}$ converges (bornologically) to zero.
\end{defn}

An analogous definition can be given for the convergence of filters of subsets of $E$. We omit the details of this definition since is not important for our scope.

\begin{rmk}
	The notion of bornological convergence on a bornological vector space of convex type $E = \underset{B \in \cB_E}\limind E_B$ can be restated in the following way: $\{ x_n \}_{n \in \N}$ is convergent to zero in the sense of Mackey if and only if there exists a $B \in \cB_E$ and $N \in \N$ such that for all $n > N$, $x_n \in E_B$ and $x_n \to 0$ in $E_B$ for the semi-norm of $E_B$.
\end{rmk}

\begin{defn}
Let $E$ be a bornological vector space over $K$.
\begin{itemize}
\item a sequence $\{x_n\}_{n \in \N} \subset E$ is called \emph{Cauchy-Mackey} if the double sequence $\{x_n - x_m\}_{n,m \in \N}$ converges to zero;
\item a subset $U \subset E$ is called \emph{(bornologically) closed} if every Mackey convergent sequence of elements of $U$ converges (bornologically) to an element of $U$. 
\end{itemize}
\end{defn}

\begin{defn}
A bornological vector space is called \emph{semi-complete} if every Cauchy-Mackey sequence is convergent.
\end{defn}

The notion of semi-completeness is not as useful as the notion of completeness in the theory of topological vector spaces. We remark that any complete bornological vector space is semi-complete, but the converse is false.

It can be shown that the notion of bornologically closed subset induces a topology on $E$, but this topology is neither a vector space topology, nor a group topology in general. Therefore, an arbitrary intersection of bornological closed subsets of a bornological vector space is bornologically closed. So, the following definition is well posed.

\begin{defn}
	Let $U \subset E$ be a subset of a bornolgical vector space. The closure of $U$ is the smallest bornologically closed subset of $E$ in which $U$ is contained. We denote the closure of $U$ by $\ol{U}$.
\end{defn}

The concept of bornologically closed subspace fits nicely in the theory. For example a bornological vector space is separated, in the sense of Definition \ref{defn:separated_born}, if and only if $\{0 \}$ is a bornologically closed subset. We have to warn the reader that the closure of a subset $X \subset E$ of a bornological subset is not always equal to the limit points of convergent sequences of elements of $X$ but strictly contains it.

\subsection*{Duality between bornologies and topologies}

Let $\bLoc_K$ denote the category of locally convex vector spaces over $K$.
We recall the definitions of two functors from \cite{H2}, ${}^t: \bBorn_K \to \bLoc_K$ and ${}^b: \bLoc_K \to \bBorn_K$. To a bornological vector space of convex type $E$ we associate the topological vector space $E^t$ in the following way: we equip the underlying vector space of $E$ with a topology for which a basis of neighborhoods of $0$ is given by \emph{bornivorous subsets}, \ie subsets that absorb all bounded subsets of $E$. The association $E \mapsto E^t$ is well defined and functorial. Then, if $E$ is a locally convex space, $E^b$ is defined to be the bornological vector space obtained by equipping the underlying vector space of $E$ with the \emph{von Neumann} (also called canonical) bornology, whose bounded subsets are the subsets of $E$ absorbed by all neighborhood of $0 \in E$. Also, this association is well defined and functorial.
In chapter 1 of \cite{H2} one can find many properties of these constructions of which the main one is that they form an adjoint pair of functors. 

We conclude this review of the theory by recalling some relations between bornological and topological vector spaces that we will use later in our proofs.

\begin{prop} \label{prop:born_convergence}
	Let $E$ be a locally convex space and consider the bornological vector space $E^b$. Then, if a sequence $\{x_n\}_{n \in \N}$ converges to $0$ in the sense of Mackey in $E^b$ it converges topologically in $E$.
\end{prop}

\begin{proof}
Let $B \subset E$ be a von Neumann bounded subset such that the condition of Mackey convergence to $0$ for $\{x_n\}_{n \in \N}$ is satisfied.
Given any neighborhood of zero $U \subset E$, then there must exist a $\la \in K^\times$ such that
\[ \la B \subset U \] 
and by Definition \ref{defn:born_conv} there exists an $n = n(\la)$ such that 
\[ x_i \in \la B \subset U, \ \ \forall i > n, \]
hence the sequence converge topologically.
\end{proof}

\begin{defn} \label{def:born_metrizable}
Let $E$ be a bornological vector space of convex type. We say that $E$ is metrizable if $E \cong F^b$ for a metrizable locally convex space. 
\end{defn}

The following result is a well-known statement of the theory of bornological vector spaces over $\R$ or $\C$, see for example proposition (3) of section 1.4.3 of \cite{H}. But as far as we know the non-Archimedean version of the result is hard to find in literature. The last lines of page 108 of \cite{H2}, essentially affirm, without proof, the statement of next proposition. Here we offer a detailed proof.

\begin{prop} \label{prop:metric}
	Let $E$ be metrizable bornological vector space of convex type. Then, a sequence $\{ x_n \}_{n \in \N}$ converges bornologically in $E$ if and only if it converges topologically in $E^t$.
\end{prop}

\begin{proof}
	Thanks to Proposition \ref{prop:born_convergence} we only need to check that topological convergence implies bornological convergence.
	So, let $\{V_n \}_{n \in \N}$ be a countable base of absolutely convex neighbourhoods of $0$ in $E$ such that $V_{n + 1} \subset V_n$ for every $n \in \N$. Let $A = \{ x_n \}_{n \in \N}$ be a sequence in $E$ which converges to $0$ topologically. 

Since the sequence $A$ converges topologically to zero, then it is
absorbed by every neighbourhood of zero. Therefore, for every $n \in \N$,
there exists a $\a_n \in |K^\times|$ such that 
\[ A \subset \la_n V_n \]
for $\la_n \in K^\times$ with $|\la_n| = \a_n$.
It follows that 
\[ A \subset \bigcap_{n = 1}^\infty \la_n V_n. \]
Let $\{\be_n \}_{n \in \N}$ be a sequence of strictly positive real numbers such that $\be_n \in |K^\times|$ and $\be_n \to 0$ as $n \to \infty$. Let $\gamma_n = \frac{\a_n}{\be_n}$ and  $\epsilon \in |K^\times|$ be given. Define
\[ B = \bigcap_{n = 1}^\infty \mu_n V_n \]
with $\mu_n \in K^\times$ and $|\mu_n| = \gamma_n$. $B$ is clearly a bounded subset of $E$, because it is absorbed by all $V_n$. We are going to prove the following assertion: for every $\epsilon \in |K^\times|$ there is an integer $m \in \N$ for which $A \cap V_m \subset \la B$, for $\la \in K^\times$ with $|\la| = \epsilon$, from which the proposition will be then proved.

Since the sequence $\frac{\gamma_n}{\a_n} = \frac{1}{\be_n}$ tends to $\infty$ for $n \to \infty$, there is an integer $p \in \N$ such
that 
\[ \forall n > p, \ \frac{\gamma_n}{\a_n} > \frac{1}{\epsilon}. \]
Therefore 
\[ A \subset \bigcap_{n = 1}^\infty \la_n V_n \then A \subset \rho_n V_n,  \]
for $\rho_n = \la \mu_n \in K^\times$ with $|\la| = \epsilon$ and $n > p$.
But the set
\[ \bigcap_{n = 1}^p \rho_n V_n, \]
with again $\rho_n = \la \mu_n \in K^\times$ and $|\la| = \epsilon$, is a neighbourhood of $0$. Hence there exists an integer $m \in \N$ such that
\[ V_m \subset \bigcap_{n = 1}^p \rho_n V_n \]
therefore
\[ A \cap V_m \subset \bigcap_{n = 1}^\infty \rho_n V_n = \bigcap_{n = 1}^\infty \la \mu_n V_n = \la B,  \]
proving the claim.
\end{proof}

The last proposition has the following interesting consequence.

\begin{cor} \label{cor:metric}
Let $E$ be a metrizable locally convex space. Then $E$ is complete topologically, if and only if $E^b$ is complete bornologically if and only if $E^b$ is semi-complete.
\end{cor}
\begin{proof}
That the completeness of $E$ implies the completeness of $E^b$ is proved in Proposition 15 of page 101 of \cite{H2}, and as we have already remarked the completeness of $E^b$ implies its semi-completeness. Then, Proposition \ref{prop:metric} directly implies that the semi-completeness of $E^b$ implies the completeness of $E$ proving the corollary.
\end{proof}

\section{The closed graph theorem for bornological spaces with nets} \label{sec:closed_nets}

In this section we will prove the Closed Graph Theorem for bornological vector spaces endowed with a net. Essentially we will adapt the proof of Popa's Theorem that one can find in \cite{H} and we will make it work over any complete, non-trivially valued field $K$.

\begin{defn} \label{defn:net}
Let $F$ be a $K$-vector space. A \emph{net} on $F$ is a map $\cN: \underset{k \in \N}\bigcup \N^k \to \sP(F)$ such that
\begin{enumerate}
	\item each $\cN(n_1, \ldots, n_k)$ is a disk;
	\item $\cN(\void) = F$;
	\item for every finite sequence $(n_0, \ldots, n_k)$ one has
	\[  \cN(n_0, \ldots, n_k) = \bigcup_{n \in \N} \cN(n_0, \ldots, n_k, n). \]
\end{enumerate}
\end{defn}

Notice that condition $(2)$ of the previous definition is used to give meaning to the formula
\[ F = \cN(\void) = \bigcup_{n \in \N} \cN(n). \]
If $s: \N \to \N$ is a sequence we will use the notation
\[ \cN_{s, k} = \cN(s(0), \ldots, s(k)). \]
 
\begin{defn} \label{defn:net_compatible}
Let $F$ be a separated bornological $K$-vector space of convex type. Then we can say that a net $\cN$ on $F$ is compatible with its bornology if
\begin{enumerate}
	\item for every sequence $s: \N \to \N$ there is a sequence of positive real numbers $b(s): \N \to \R_{> 0}$ such that for all $x_k \in \cN_{s,k}$ and $a_k \in K$ with $|a_k| \le b(s)_k$ the series
	\[ \sum_{k \in \N} a_k x_k \]
	converges bornologically in $F$ and $\underset{k \ge n}\sum a_k x_k \in \cN_{s,n}$ for every $n \in \N$.
	\item For every sequences $\{ \la_k\}_{k \in \N}$ of elements of $K$ and $s: \N \to \N$ the subsets
	\[ \bigcap_{k \in \N} \la_k \cN_{s, k} \]
	are bounded in $F$.
\end{enumerate}

We say that a separated bornological vector space \emph{has a net} if there exists a net which is compatible with its bornology. 
\end{defn}

The most common bornological vector spaces used in functional analysis have nets.

\begin{exa} \label{exa:nets}
\begin{enumerate}
\item Let $F$ be a bornological Fr\'echet space. By this we mean that $F \cong E^b$ for a Fr\'echet space $E$. Then $F$ has a net compatible with its bornology. To show this, consider a base of neighborhoods $\{ V_n \}_{n \in \N}$ of $0 \in F$. If the base field $K$ is Archimedean it is well known that one can define a net in the following way. For any $k$-tuple of integers $(n_1, \ldots, n_k)$ define
\[ \cN(n_1, \ldots, n_k) = n_1 V_{n_1} \cap \ldots \cap n_k V_{n_k}. \]
See section 4.4.4 of \cite{H} for a detailed proof of this fact.
This definition does not work for non-Archimedean base fields since $|n_k| \le 1$ and so it does not satisfy condition (3) of Definition \ref{defn:net}. So, let's pick an element $\a \in K$ such that $|\a| > 1$. This is always possible since $K$ is non-trivially valued. Then, let's define
\[ \cN(n_1, \ldots, n_k) = \a^{n_1} V_{n_1} \cap \ldots \cap \a^{n_k} V_{n_k}  \]
and check that this is a net compatible with the bornology of $F$. Since every neighborhood of $0$ is absorbent then the condition of $\cN$ to be a net is clearly satisfied. Let $\{ n_k \}_{k \in \N}$  be a sequence of integers, define $v_k = |\a^{ -2^k n_k}|$ and choose $\la_k \in K$ with $|\la_k| \le v_k$. Then, the series
\[ \sum_{k = 0}^\infty \la_k x_k \]
converges bornologically in $F$ for every $x_k \in \cN(n_1, \ldots, n_k)$ because it converges for the metric of $F$, and by Proposition \ref{prop:metric} the convergence for the metric of $F$ is equivalent to the bornological convergence in $F$. Moreover, for every $k_0$
\[  \sum_{k = k_0}^\infty \la_k x_k \in \la_{k_0} \cN(n_1, \ldots, n_k) \subset \cN(n_1, \ldots, n_k),  \]
because $|\la_k| < 1$ and for non-Archimedean base fields disks are additive sub-groups of $F$. Then, we need to check the second condition on compatibility of the net with the bornology, which is pretty easy to check since given any sequence $\{\la_k\}_{k \in \N}$ of elements of $K$, we can consider the set 
\[ A = \bigcap_{k = 0}^\infty \la_k \cN(n_1, \ldots, n_k) =  \bigcap_{k = 0}^\infty \la_k \a^{n_1} V_{n_1} \cap \ldots \cap \a^{n_k} V_{n_k}. \]
Thus, given any $V_{n_r}$, we have that
\[  A \subset \la_r \a^{n_r}  V_{n_r}  \]
hence $A$ is absorbed by any neighborhood of $0 \in F$, hence it is bounded for the von Neumann bornology of the Fr\'echet space.

\item Let $F = \underset{n \in \N}\limind F_n$ be a monomorphic inductive limit such that $F_n$ are separated bornological vector spaces which have nets compatible with the bornology, then we claim that $F$ is a separated bornological vector space which has a compatible net. To show this, first we note that since the inductive limit is monomorphic 
\[ F = \bigcup_{n \in \N} F_n \]
and $F$ is separated when equipped with the direct limit bornology, see Proposition 6 at page 49 of \cite{H2} for a proof of this fact. Then, let $\cN_n$ be a net for $F_n$. Let's define 
\[ \cE(n) = \cN_1(n_1), \ \ \cE(n_1, \ldots, n_k) = \cN_{n_1}(n_2, \ldots, n_k).  \]
One can check directly that $\cE$ is a net on $F$ compatible with the bornology of $F$.

\item From the previous example it follows that every complete bornological vector space with countable base for its bornology has a net compatible with the bornology. In particular regular LB spaces and regular LF spaces have nets for their von Neumann bornology.

\item As an example of bornological vector space which cannot be endowed with a net, one can consider an infinite dimensional Banach space $E$ endowed with the bornology of pre-compact subsets, if the base field is locally compact, or the bornology of compactoid subsets if the base field is not locally compact. Let's denote $E^c$ the vector space $E$ equipped with this bornology, which is of convex type and complete (see examples 1.3 (9) and (10) for a proof of this fact for Archimedean base fields. The same argument works for any base field). It is also well known that the identity map $E^c \to E^b$ is bounded, but the two bornologies do not coincide in general. This also means that the bornology of $E^c$ cannot have a net because the fact that the identity map is bounded implies that, if $E^c$ would have a net, we could apply the Isomorphism Theorem \ref{thm:iso} to deduce that the identity map is bounded also in the other direction $E^b \to E^c$. Notice that, by previous examples, this also shows that the bornology of $E^c$ has not a countable base.

\end{enumerate}
\end{exa}

Let's move towards the proof of the Closed Graph Theorem, and before proving it let's prove some lemmata.

\begin{lemma} \label{lemma:absorbing}
	Let $E$ be and $F$ be $K$-vector spaces. Let $B \subset E$ be bounded, $C \subset F$ any subset and $f: E \to F$ a linear map. Then, $C$ absorbs $f(B)$ if and only if $f^{-1}(C)$ absorbs $B$.
\end{lemma}
\begin{proof}
This is a very basic property of linear maps. $C$ absorbs $f(B)$ means that there exists a $\la \in K^\times$ such that
\[ f(B) \subset \la C \]
hence
\[ B \subset f^{-1}(f(B)) \subset  f^{-1}(\la C) = \la f^{-1}(C). \]
On the other hand, if $B \subset \la f^{-1}(C)$ then
\[ f(B) \subset f( \la f^{-1}(C)) \subset \la  f( f^{-1}(C)) \subset \la C. \]
\end{proof}

\begin{lemma} \label{lemma:baire}
	Let $E$ be a $K$-Banach space and $F$ be a separated convex bornological vector space endowed with a net $\cN$, not necessarily supposed to be compatible with the bornology. Let $f: E \to F$ be a linear map, then there exists a sequence $\{ n_k \}_{k \in \N}$ of integers such that $f^{-1}(\cN(n_1, \ldots, n_k))$ is not meagre in $E$ for each $k \in \N$.
\end{lemma}
\begin{proof}
By the definition of net $F = \underset{n_1 \in \N}\bigcup \cN(n_1)$, so $E = \underset{n_1 \in \N}\bigcup f^{-1}(\cN(n_1))$ and since $E$ is a Baire space, it follows that there must exist a $n_1$ such that $f^{-1}(\cN(n_1))$ is not meagre in $E$. Then, we can apply the same reasoning to the relation $\cN(n_1) = \underset{n_2 \in \N}\bigcup \cN(n_1, n_2)$ obtaining a $f^{-1}(\cN(n_1, n_2))$, for some $n_2 \in \N$, which is not meagre in $E$, and inductively for any $k$ we get a $f^{-1}(\cN(n_1, \ldots, n_k))$ which is not meagre in $E$, for suitable $n_1, \ldots, n_k \in \N$.
\end{proof}

The next lemma is the key technical lemma to prove the Closed Graph Theorem for bornological spaces equipped with a net.

\begin{lemma} \label{lemma:net}
	Let $E$ be a $K$-Banach space and $F$ be a separated convex bornological vector space endowed with a net $\cN$ compatible with its bornology. Let $B \subset E$ denotes the open unit ball of $E$. If $f: E \to F$ is a morphism with bornologically closed graph, then there exist a sequence $(n_k)$ of integers such that $f(B)$ is absorbed in each $\cN(n_1, \ldots , n_k)$.
\end{lemma}
\begin{proof}
By Lemma \ref{lemma:baire} we can produce a sequence $\{ n_k \}_{k \in \N}$ of integers such that $f^{-1}(\cN(n_1, \ldots, n_{k_0}))$ is not meagre in $E$ for each $k_0 \in \N$. It is sufficient to show that for each fixed $k_0$ the set $f^{-1}(\cN(n_1, \ldots, n_{k_0}))$ absorbs $B$ by Lemma \ref{lemma:absorbing}.

Let's denote by $s: \N \to \N$ the sequence obtained by applying Lemma \ref{lemma:baire}. By the compatibility of $\cN$ with the bornology on $F$ there exists a sequence $b(s)$ of positive real numbers such that for each sequence $a_k$ of elements of $K$ with $|a_k| \le b(s)_k$ and for all $x_k \in \cN_{s, k}$, the series $\underset{k \ge k_0}\sum a_k x_k$ converges bornologically to an element in $\cN_{s, k_0}$.
Let $\epsilon > 0$, we can choose $a_k$ such that
\[ \sum_{k \in \N} |a_k| \le \epsilon. \]
Let's denote $A_k = a_k f^{-1}(\cN_{s, k})$. Since $A_k$ is not meagre, then there is a point $b_k$ in the interior of $\ol{A_k}$ and a radius $\rho_k > 0$ such that the open ball 
\[ D(b_k, \rho_k) = b_k + \mu_k B, \ \ \ \mu_k \in K, \ \ \ |\mu_k| = \rho_k, \]
of radius $\rho_k$ and centred in $b_k$ is contained in $\ol{A_k}$. We can assume that $b_k \in A_k$. In fact, since $b_k \in \ol{A_k}$, then there exists $b_k' \in A_k$ such that $|b_k - b_k'|_E < \frac{\rho_k}{2}$. So,
\[ b_k' + D(0, \frac{\rho_k}{2}) = (b_k' - b_k) + (b_k + D(0, \frac{\rho_k}{2})) \subset \ol{A_k}.  \]
We may also suppose that $\rho_k \le \frac{1}{k}$ without loss of generality. So for a fixed $k_0$, we have 
\[ D(0, \rho_{k_0}) \subset \ol{A_{k_0}} - b_{k_0} \subset \ol{A_{k_0}} - \ol{A_{k_0}}.  \]
Thus, if $K$ is Archimedean we can deduce that
\[ D(0, \rho_{k_0}) \subset 2 \ol{A_{k_0}} \]
and if $K$ is non-Archimedean that
\[ D(0, \rho_{k_0}) \subset \ol{A_{k_0}}, \]
hence if we redefine $\rho_{k_0}$ being half its value when $K$ is Archimedean we can always suppose that
\[ D(0, \rho_{k_0}) \subset \ol{A_{k_0}}. \]
We will conclude the proof by showing that there exists $\gamma \in K^\times$, such that
\[ \ol{f^{-1}(\cN_{s, k_0})} \subset \gamma f^{-1}(\cN_{s, k_0})) \]
because then we can deduce that 
\[ \mu_{k_0} B = D(0, \rho_{k_0}) \subset \gamma A_{k_0} \then B \subset \mu_{k_0}^{-1} \gamma a_{k_0} \cN(n_1, \ldots, n_{k_0}),  \ \ \mu_{k_0} \in K, \ \ |\mu_{k_0}| = \rho_{k_0}. \]
Pick an element $x \in \ol{f^{-1}(\cN_{s, k_0})}$. There exists an element $y_{k_0} \in f^{-1}(\cN_{s, k_0})$ such that $|x - y_{k_0}|_E \le \rho_{k_0 + 1}$.  Therefore
\[ (y_{k_0} - x) + b_{k_0 + 1} \in D(b_{k_0 + 1}, \rho_{k_0 + 1}) \subset \ol{A_{k_0 + 1}}. \]
Then, we can find $y_{k_0 + 1} \in A_{k_0 + 1}$ such that $|x - y_{k_0}  - y_{k_0 + 1} + b_{k_0 + 1}|_E \le \rho_{k_0 + 2}$. So, by induction for every $N > k_0$ we can find elements $y_k \in \ol{A_k}$ such that
\[ |x - \sum_{k \ge k_0}^N y_k  + \sum_{k \ge k_0 + 1}^N b_k |_E \le \rho_{N + 1}  \]
Since $\rho_N \to 0$, the left-hand side converges to $0$. Let's show that the series 
\[ \sum_{k \ge k_0}^N z_k  + \sum_{k \ge k_0 + 1}^N c_k \]
where $z_k = f(y_k)$ and $c_k = f(b_k)$ converges to $f(x)$. By hypothesis $z_{k_0} \in \cN_{s, k_0}$ and $z_k, c_k \in a_k \cN_{s, k}$ for $k > k_0$, so by the compatibility of the net with the bornology of $F$ the series
\[ \sum_{k \ge k_0}^\infty z_k , \ \  \sum_{k \ge k_0 + 1}^\infty c_k \]
converge bornologically in $F$. Moreover, since $\cN_{s, k} \subset \cN_{s, k_0}$ for each $k > k_0$, one has that
\[ \sum_{k \ge k_0}^\infty z_k \in \cN_{s, k_0} + \gamma' \cN_{s, k_0} \]
\[ \sum_{k \ge k_0 + 1}^\infty c_k  \in \gamma' \cN_{s, k_0} \]
for a $\gamma' \in K$ such that $|\gamma'| \le \epsilon$. So
\[ y = \sum_{k \ge k_0}^\infty z_k - \sum_{k \ge k_0 + 1}^\infty c_k   \in \gamma \cN_{s, k_0} \]
where $\gamma \in K^\times$ can be chosen to have absolute value $1 + 2 \epsilon$ if $K$ is Archimedean and $1$ if $K$ is non-Archimedean, because $\cN_{s, k_0} + \gamma' \cN_{s, k_0} = \cN_{s, k_0}$ for $|\gamma'| < 1$ in this case.
Since the graph of $f$ is bornologically closed in $E \times F$, then we have that
\[ 0 = f(0) = f(x - \sum_{k \ge k_0}^\infty y_k  + \sum_{k \ge k_0 + 1}^\infty b_k) = f(x) - f(\sum_{k \ge k_0}^\infty y_k  + \sum_{k \ge k_0 + 1}^\infty b_k) = f(x) - y \]
so $f(x) \in \gamma \cN_{s, k_0}$ which implies that $x \in \gamma f^{-1}(\cN_{s, k_0})$, thus proving the lemma.
\end{proof}

\begin{thm} \label{thm:net}
	Let $E$ and $F$ be separated convex bornological vector spaces, where $E$ is complete and $F$ has a net $\cN$ compatible with its bornology. Then every linear map $f: E \to F$ with bornologically closed graph is bounded.
\end{thm}

\begin{proof}
Notice that $E \cong \underset{B \in \cB_E}\limind E_B$, where $B$ runs through all bounded Banach disks, i.e. $E$ can be described as a monomorphic inductive limit of a family of Banach spaces. In order to prove that $f$ is bounded we only have to show that the compositions of $f$ with the canonical maps $i_B: E_B \to E$ are bounded. Indeed, the
graph of $f \circ i_B$ is bornologically closed, so we just need to show the theorem for $E$ supposed to a Banach space. To see that the graph of $f \circ i_B$ is closed one can consider a sequence $\{x_n\} \subset E_B$, converging to $x \in E$. Since $i_B$ is bounded then $i_B(x_n) \to i(x)$, in the sense of Makey, so $f(i_B(x_n)) \to f(i(x))$ in the sense of Mackey in $F$, because the graph of $f$ is bornologically closed. 

Therefore, suppose that $E$ is a $K$-Banach space with unit ball $B \subset E$.
By Lemma \ref{lemma:net} there exists a sequence $\{ n_k \}_{k \in \N}$ of integers such that $f(B)$ is absorbed in each $\cN(n_1, \ldots , n_k)$. It follows that there exists a sequence $\{ a_k \}_{k \in \N}$ of elements of $k$ such that 
\[ f(B) \subset \bigcap_{k \in \N} a_k \cN(n_1, \ldots, n_k) \]
and the latter subset is bounded in $F$, by the request of the net to be compatible with the bornology of $F$. So, we can conclude that $f(B)$ is bounded in $F$.
\end{proof}

We will discuss some applications of this theorem in the last section. Let's now see some stability properties of bornological nets, for which there is not much in literature known to the author.

\begin{prop} \label{prop:stability}
	Bornological nets have the following stability properties:
	\begin{enumerate}
		\item if $E = \underset{i \in \N}\limind E_i$ is a monomorphic direct limit of bornological vector spaces with nets then $E$ has a net;
		\item if $(E, \cB)$ has a net and $\cB'$ is another bornology on $E$ such that the identity $(E, \cB) \to (E, \cB')$ is bounded, then $(E, \cB')$ has a net;
		\item every closed subspace $F \subset E$ of a bornological vector space with a net has a net;
		\item if $E = \underset{i \in \N}\limpro E_i$ is the countable projective limits of bornological spaces with nets, then $E$ has a net;
		\item if $E = \underset{i \in \N}\bigoplus E_i$ is the countable coproduct of bornological spaces with nets, then $E$ has a net.
	\end{enumerate}
\end{prop}

\begin{proof}
\begin{enumerate}
\item This claim has already been discussed in Example \ref{exa:nets} (3).

\item Indeed, one can use the same given net on $(E, \cB)$ on the space $(E, \cB')$, which is easily seen to be compatible with $\cB'$ too.

\item If $E$ is equipped with a net $\cN$ compatible with its bornology and $F \subset E$ is bornologically closed, then the association 
\[ \cN_F(n_1, \ldots, n_k) = \cN(n_1, \ldots, n_k) \cap F \]
defines a net on $F$ which is compatible with the bornololgy induced by $E$ on $F$ because $F$ is bornologically closed.

\item The projective limit of a diagram $I \to \bBorn_K$ is the linear subspace
\[ P = \left \{ (x_i) \in \prod_{i \in I} E_i | f_{i, j}(x_j) = x_i \text{ for all } f_{i,j} \right \}, \]
endowed with the induced bornology. It is easy to check that if all spaces $E_i$ are separated then $P$ is bornologically closed in $\underset{i \in I}\prod E_i$. Thus, by the previous point it is enough to show that the product of any countable family of bornological vector spaces with nets is a bornological vector space with a net. So, let $\{ (E_i, \cN^{(i)}) \}_{i \in \N}$ be a countable family of bornological vector spaces endowed with nets. 

The product bornology on $\underset{i \in \N}\prod E_i$ is separated, see Proposition 5 on page 48 of \cite{H2}. For any $n$ let's fix bijections $f_n: \N^{n+1} \to \N$ and let $s: \N \to \N$ be any given sequence. For any $k \in \N$ there exists $(a_0^{(k)}, \ldots, a_{k}^{(k)}) \in \N^{k + 1}$ such that 
\[ s(k) = f_k(a_0^{(k)}, \ldots, a_{k}^{(k)}) \]
So, we define a family of sequences $\{ s_n : \N \to \N \}_{n \in \N}$ by
\[ s_n(k) = a_n^{(k + n)} \]
for each $k \in \N$. We set
\[ \cN_{s, k} \doteq \cN_{s_0, k}^{(0)} \times \cN_{s_1, k - 1}^{(1)} \times \ldots \times \cN_{s_k, 0}^{(k)} \times \prod_{i > k} E_i \]
to get a well-defined map $\cN: \underset{i \in \N}\bigcup \N^i \to \sP( \underset{i \in \N}\prod E_i)$, by setting $\cN(\void) = \underset{i \in \N}\prod E_i$ and $\cN(n_1, \ldots, n_k) = \cN_{s, k}$ when we choose $s: \N \to \N$ such that $s(i) = n_i$ for each $i \in \N$.

Let's check that $\cN$ is a net on $\underset{i \in \N}\prod E_i$. The properties (1) and (2) of Definition \ref{defn:net} are obvious. So let's check the last condition.

First, let's fix for any $n$ a sequence $t^{(n)}: \N \to \N$ such that $t^{(n)}(i) = s(i) = n_i$ for all $i \le k$, $t^{(n)}(k + 1) = n$, and the other values of $t^{(n)}$ can be freely chosen. Then, consider
\[ x = (x_k)_{k \in \N} \in \cN_{s, k} \times \prod_{i > k} E_i. \]
We have to show that there exists an $n$ such that
\[ x \in \cN_{t^{(n)}, k + 1} \times \prod_{i > k + 1} E_i. \]
Then, for any $n$ 
\[ \cN_{t^{(n)}, k + 1} = \cN_{s_0, k + 1}^{(0)} \times \cN_{s_1, k }^{(1)} \times \ldots \times \cN_{s_k, 1}^{(k)} \times \cN_{t^{(n)}_k, 0}^{(k + 1)} \]
so that the first $k$ components do not change changing $n$, whereas
\[ \cN_{t^{(n)}_k, 0}^{(k + 1)} = \cN^{(k + 1)} ((a_n)_k^{(k)}) \]
for the $(k + 1)$-tuple such $f_k((a_n)_0^{(k)}, \ldots, (a_n)_k^{(k)}) = n$. Thus, since $f_k$ is a bijection then $\underset{n \to \infty}\limsup (a_n)_k^{(k)} \to \infty$ and since $\cN^{(k + 1)}$ is a net we have that
\[ \bigcup_{n \in \N} \cN^{(k + 1)} ((a_n)_k^{(k)}) = E_{k + 1}, \]
which shows (3) of Definition \ref{defn:net}.

Then, let's prove that $\cN$ is compatible with the bornology of $\underset{i \in \N}\prod E_i$. For any $s: \N \to \N$ let's define
\[ b(s)_k \doteq  \min \{ b^{(i)} (s_i)_{k - i} | 0 \le i \le k \} \]
where $b^{(i)}$ is as in Definition \ref{defn:net_compatible} for $\cN^{(i)}$. Let's fix a sequence $\mu_k$ of elements of $K$ such that $|\mu_k| \le b(s)_k$. It is an easy consequence of the definition of product bornology that a series $\underset{k \in \N}\sum \mu_k x_k$, with $x_k \in \cN_{s, k}$, converges in $\underset{i \in \N}\prod E_i$ if and only if its components converge in $E_i$ for each $i \in \N$. Let's denote with $\pi_n: \underset{i \in \N}\prod E_i \to E_n$ the canonical projection, so for $l > n$
\[ \pi_n \left ( \sum_{k = 0}^l \mu_k x_k \right ) = \] 
\[ = \sum_{k = 0}^{n - 1} \mu_k \pi_n(x_k) + \sum_{k = n}^{l} \mu_{(k - n)}^{(n)} \frac{\mu_k}{\mu^{(n)}_{(k - n)}}  \pi_n(x_k) \]
where $\mu_{(k - n)}^{(n)} \in K$ is such that $|\mu_{(k - n)}^{(n)}| \le b^{(n)}(s_n)_{(k - n)}$. And notice that we can always have that
\[ \left | \frac{\mu_k}{\mu^{(n)}_{(k - n)}} \right | \le 1, \ \ \forall k \ge n. \]
Hence the series 
\[ \sum_{k = n}^{\infty} \mu^{(n)}_{(k - n)} \frac{\mu_k}{\mu^{(n)}_{(k - n)}}  \pi_n(x_k) \]
converges because $\cN^{(n)}$ is a net on $E_n$. 

Then, we need to check that for any given $m \in \N$ we have that
\[ \sum_{k \ge m} \mu_k x_k \in \cN_{s, m} \]
and since 
\[ \cN_{s, m} = \cN_{s_0, m}^{(0)} \times \cN_{s_1, m - 1}^{(1)} \times \ldots \times \cN_{s_m, 0}^{(m)} \times \prod_{i > m} E_i \]
we need to check only the first $m + 1$ components. So, we need to check that
\[ \pi_n \left ( \sum_{k = m}^\infty \mu_k x_k \right ) \in \cN_{s_n, m - n}^{(n)} \]
for $0 \le n \le m$. But for the above formula
\[ \pi_n \left ( \sum_{k = m}^\infty \mu_k x_k \right ) = \sum_{k = m}^{\infty} \mu_{(k - n)}^{(n)} \frac{\mu_k}{\mu^{(n)}_{(k - n)}}  \pi_n(x_k) \]
which belongs to $\cN_{s_n, m - n}^{(n)}$ because $\cN^{(n)}$ is a net on $E_n$ compatible with its bornology.

Last condition of Definition \ref{defn:net_compatible} is easy to check because a subset $B \subset \underset{i \in \N}\prod E_ni$ is bounded if and only if $\pi_i(B)$ is bounded in $E_i$.

\item Finally the last assertion of the proposition is a consequence of (1) and (4), because for any family of bornological vector spaces $\{ E_n \}_{n \in \N}$ one has the isomorphism
\[ \bigoplus_{n \in \N} E_n \cong \limind_{n \in \N} \prod_{i = 1}^n E_i \]
and the inductive system is monomorphic.
\end{enumerate}
\end{proof}

It is easy to see that the quotient of a bornological vector space endowed with a net is not necessarily endowed with a net. One can consider a Fr\'echet-Montel space $E$ and a quotient $E/F$, for a closed subspace $F \subset E$ which is not Fr\'echet-Montel. Examples of these kind of spaces are well known both for Archimedean and non-Archimedean $K$. So, the von Neumann bornology of $E$ coincides with the compact(oid) bornology of $E$, but the quotient bornology of $E/F$ is the compact(oid) one and does not coincide with the von Neumann one. Example \ref{exa:nets} (4) implies that $E/F$ does not admit a net compatible with its bornology.

\section{The closed graph theorem for bornological spaces with webs} \label{sec:closed_webs}

In this section we will prove the most general version of the Closed Graph Theorem for bornological vector space of the paper.

\begin{defn} \label{defn:born_web}
Let $E$ be a separated bornological vector space of convex type over $K$. A pair $(\cV, b)$ consisting of mappings $\cV : \underset{k \in \N}\bigcup \N^k \to \mathcal{P}(E)$ and $b : \N^\N \to (|K^\times|)^\N$ is called a \emph{bornological web} if all the conditions below hold:
\begin{enumerate}
\item The image of $\cV$ consists of disks.
\item $\cV(\void) = E$.
\item Given a finite sequence $(n_0, \dots, n_k)$, then $\cV(n_0, \dots, n_k)$ is absorbed by
      \[ \bigcup_{n \in \N} \cV(n_0, \dots, n_k, n). \]
\item For every $s: \N \to \N$ the series $\underset{k \in \N}\sum \la(s)_k x_k$, with $\la(s) \in K$, converges bornologically in $E$, whenever we choose $x_k \in \cV(s(0), \dots ,s(k))$ and $|\la(s)_k| = b(s)_k$.
\end{enumerate}
\end{defn}

As we did in previous section we will use the shorthand notation $\cV_{s,k} = \cV(s(0), \ldots , s(k))$. We define the following sets, which depend on $b$:
\[ 
\forall s: \N \to \N, \forall n \in \N : \wtilde{\cV}_{s,n} = \Gamma( \cV_{s,n} \cup \l \{ \sum_{k \ge n +1} \la(s)_k x_k | \forall k \ge n + 1 : x_k \in \cV_{s,k}, |\la(s)_k| = b(s)_k \r \} ) \]
where $\Gamma$ denotes the absolutely convex hull. Furthermore, let $\cB_{(\cV,b)}$ denote the convex linear bornology on $E$ which is generated by all subsets of the form 
\[ \bigcap_{k \in \N}  \mu_k \wtilde{\cV}_{s,k}, \]
where the $\{\mu_k \}_{k \in \N}$ is an arbitrary $K$-valued sequence.

\begin{defn}
A separated bornological vector space of convex type $E$ which is endowed with a bornology of the form $\cB_{(\cV, b)}$ for a bornological web $(\cV, b)$ on $E$ is called a \emph{webbed convex bornological space}.
\end{defn}

The following is the Closed Graph Theorem for webbed bornological vector spaces.

\begin{thm} \label{thm:web}
Let $E$ and $F$ be separated convex bornological vector spaces, where $E$ is complete and $F$ is endowed with a bornological web $(\cV, b)$. Then, every linear map $f: E \to F$ with bornologically closed graph is bounded for the bornology $\cB_{(\cV, b)}$.
\end{thm}

\begin{proof}
First, as in Theorem \ref{thm:net} we can reduce the proof to the case in which $E$ is a Banach space, because $f: E = \limind E_i \to F$ is bounded if and only if $f \circ \a_i$ is bounded for every $i$, where $\a_i: E_i \to \limind E_i$ are the canonical maps.

By condition (3) of Definition \ref{defn:born_web} we can use a reasoning similar to the one given in Lemma \ref{lemma:baire} to produce a sequence $s: \N \to \N$ such that $f^{-1}(\cV_{s, k})$ is not meagre in $E$ for any $k$. Moreover let's put $b_k = b(s)_k$, for all $k \in \N$. Since $(\cV, b)$ satisfies condition (4) of Definition \ref{defn:born_web}, the series $\underset{k \in \N}\sum \la_k x_k$ converges bornologically in $F$,
whenever we choose $x_k \in \cV_k$ and $\la_k \in K$ with $|\la_k| = b_k$.

Next, let $D(r)$ denote the ball of radius $r$ in $E$ centred in zero. If we can show that $f (D(1))$ is absorbed by $\wtilde{\cV}_{s, k}$ , or equivalently by Lemma \ref{lemma:absorbing}, that $D(1)$ is absorbed by $f^{-1} ( \wtilde{\cV}_{s, k} )$, for all $k \in \N$, then $f(D(1)) \in \cB_{(\cV,b)}$, and we are done.

Define $A_k = \la_k f^{-1} (\cV_{s, k} )$, for all $k \in \N$ and pick $\la_k \in K$ with $|\la_k| = b_k$. Since $A_k$ is not meagre and consequently not nowhere dense, the interior of $\ol{A}_k$ is not empty. Hence there exist $\ol{y}_k \in \ol{A}_k$ and $\rho_k < \frac{1}{k+1}$ such that $\ol{y}_k + D(2 \rho_k) \subset \ol{A}_k$. Since $\ol{y}_k \in \ol{A}_k$, there is a $y_k \in A_k$ such that $y_k \in \ol{y}_k + D(\rho_k)$, thus
\[ y_k + D(\rho_k) = (y_k - \ol{y}_k) + (\ol{y}_k + D(\rho_k)) \subset \ol{y}_k + D(2 \rho_k) \subset \ol{A}_k. \]

So, $D(\rho_k) \subset \ol{A}_k - y_k$ which implies $B(\rho_k) \subset 2 \ol{A}_k$ if $K$ is Archimedean and $D(\rho_k) \subset \ol{A}_k$ if $K$ is non-Archimedean. 

So fix $n \in \N$ and let $x \in \ol{f^{-1} ( \wtilde{\cV}_{s,n} )}.$ Then there is a $u_n \in f^{-1} ( \wtilde{\cV}_{s,n} )$ with $x - u_n \in D(\rho_{n +1}).$
\[ x - u_n + y_{n+1} \in  D(\rho_{n +1}) + y_{n+1} \subset  \ol{A}_{n+1}. \]
So, there is a $u_{n+1} \in A_{n+1}$ with $(x - u_n + y_{n+1}) - u_{n+1} \in D(\rho_{n +2})$ and inductively we find $u_k \in A_k$, $k > n$, such that we have 
\[ x - \sum_{k=n}^l u_k + \sum_{k=n+1}^l y_k \in D(\rho_{l +  1}), \] 
for $l > n$. Hence, the series $x - \underset{k = n}{\overset{\infty}\sum} u_k + \underset{k = n + 1}{\overset{\infty}\sum} y_k$ converges to $0$, since
$\rho_{l +1} \to 0$. Define $v_k = f(u_k )$ and $z_k = f (y_k )$. Then $v_n \in \cV_{s, n}$, $z_n \in \la_n \cV_{s,n}$, and $\forall k > n$ one has that  $v_k , z_k \in \la_k \cV_{s,k}$. It follows from (4) of Definition \ref{defn:born_web} that $\underset{k \in \N}\sum v_k$ and $\underset{k \in \N}\sum z_k$ converge
bornologically in $F$ and moreover
\[ y = \sum_{k \ge n} v_k  - \sum_{k \ge n + 1} z_k = v_n + \sum_{k \ge n + 1} v_k - \sum_{k \ge n + 1} z_k \in \wtilde{\cV}_{s, n} + \wtilde{\cV}_{s, n} - \wtilde{\cV}_{s, n}.  \]
Then, since $f$ has bornologically closed graph, we infer $0 = f (0) = f (x) - y$, i.e. $f (x) = y$ which shows that
\[ f(x)  \in f(\wtilde{\cV}_{s,n}) \]
if $K$ is non-Archimedean and 
\[ f(x)  \in 3 f(\wtilde{\cV}_{s,n}) \]
if $K$ is Archimedean.

Therefore, we can deduce that 
\[  D(\rho_k) \subset 2 \ol{A}_k = 2 \ol{\la_k f^{-1} (\cV_{s,k} )} \subset 2 \ol{\la_k f^{-1} (\wtilde{\cV}_{s,k} )} \subset 6 \la_k f^{-1} (\wtilde{\cV}_{s,k} ) \]
if $K$ is Archimedean or 
\[ D(\rho_k) \subset  \ol{A}_k =  \ol{\la_k f^{-1} (\cV_k )} \subset  \ol{\la_k f^{-1} (\wtilde{\cV}_{s,k} )} \subset  \la_k f^{-1} (\wtilde{\cV}_{s,k} ) \]
if $K$ is non-Archimedean, completing the proof.
\end{proof}

We conclude this section by proving some stability properties of bornological webs.

\begin{prop}
	Bornological webs have the following stability properties:
	\begin{enumerate}
		\item if $E = \underset{i \in \N}\limind E_i$ is a monomorphic direct limit of webbed bornological vector spaces then $E$ is webbed;
		\item if $(E, \cB)$ is webbed and $\cB'$ is another bornology on $E$ such that the identity $(E, \cB) \to (E, \cB')$ is bounded, then $(E, \cB')$ is webbed;
		\item every closed subspace of a webbed bornological vector space is a webbed bornological vector space;
		\item every countable projective limit of webbed bornological vector spaces is a webbed bornological space;
		\item every countable coproduct of webbed bornological vector spaces is a  webbed bornological vector space.
	\end{enumerate}
\end{prop}

\begin{proof}
	The proof of this proposition is similar to the proofs of Proposition \ref{prop:stability}. For details we refer to \cite{G}, Theorem 4.11 where a full proof is given when the base field is Archimedean.
\end{proof}

\section{Applications}

In this last section we deduce some consequences from the theorems we have proved so far. We start by discussing the more classical ones: various forms of Isomorphisms Theorems and then we deduce de Wilde's Theorem for arbitrary base field. We conclude showing some applications to the theory of bornological algebras from \cite{Bam} and \cite{BaBeKr}. We remark, that one of the main differences in this exposition with respect to other works in literature, is that we work over any complete non-trivially valued field $K$, treating on the same footing, for as much as it is possibile, the Archimedean and the non-Archimedean sides of the theory.

\subsection*{Isomorphism theorems}

\begin{thm} \label{thm:iso}
	Let $f: E \to F$ be a bijective bounded morphism between separated bornological vector spaces with $F$ complete and $E$ with a net compatible with the bornology (resp. is webbed), then $F$ is an isomorphism of bornological vector spaces.
\end{thm}
\begin{proof}
The map	$f^{-1}: F \to E$ is a linear map between bornological vector spaces whose domain is complete and whose codomain has a net compatible with the bornology (resp. is webbed). Then, the graph of $f^{-1}$ coincides with the graph of $f$, up to swap domain with the codomain, thus is a closed subset of $F \times E$. So, by Theorem \ref{thm:net} (resp. Theorem \ref{thm:web}) $f^{-1}$ is a bounded map.
\end{proof}

For any locally convex space $E$ let's denote with $(E, \cB_\Ban)$ the vector space $E$ endowed with the vector space bornology of convex type on $E$ generated by all bounded Banach disks of $E$ and $E_\uborn = (E, \cB_\Ban)^t$, where ${}^t$ is the functor which associate to every bornological vector space the topological vector space identified by the bornivorous subset. $E_\uborn$ is called the \emph{ultrabornologification} of $E$.

\begin{defn} \label{defn:born_ultra}
Let $E$ be a locally convex space over $K$. $E$ is called \emph{bornological} if $E \cong (E^b)^t$. $E$ is called \emph{ultrabornological} if $E \cong E_\uborn$.
\end{defn}

\begin{prop}
	Let $f: E \to F$ be a bijective continuous morphism between locally convex spaces. Suppose that $E$ ultrabornological with $E^b$ endowed with a net of webbed, $F$ is bornological and $F^b$ complete. Then, $f^{-1}$ is continuous.
\end{prop}
\begin{proof}
	Direct consequence of Theorem \ref{thm:iso}, Definition \ref{defn:born_ultra} and the fact that the functors ${}^b$ and ${}^t$ are adjoints.
\end{proof}

\begin{cor}
	Let $f: E \to F$ be a bijective continuous morphism $E$ and $F$ Fr\'echet then $f$ is isomorphism.
\end{cor}
\begin{proof}
	Fr\'echet spaces are bornological and ultrabornological.
\end{proof}

\begin{rmk}
Although one can use the Closed Graph Theorem to deduce the Isomorphism Theorems, it cannot be used to deduce Open Mapping Theorems for bornological spaces, \ie that under some hypothesis a surjective bounded map must be a quotient map.
\end{rmk}

We conclude this section discussing a result that is not a consequence of the Bornological Closed Graph Theorems we are discussing, but for which we think it is important to have a proof that extends the classical one given over $\R$ and $\C$ to any valued base field. This result is Buchwalter’s theorem, which is an analogous of the Open Mapping Theorem for bornological space. The interest for this result rely on the fact that such kind of results are very rare for bornological spaces. We need a definition and a couple of lemmata, which are adaptations of \cite{WAEL2}, section 1.5.

\begin{defn} \label{defn:compatible_completant}
Let $E$ be a vector space, and $B_1$, $B_2$ be two Banach disks of $E$. We say that $B_1$ and $B_2$ are \emph{compatible} if their intersection is a Banach disk.
\end{defn}

\begin{lemma} \label{lemma:banach_disks}
Two Banach disks $B_1$ and $B_2$ of a vector space $E$ are compatible if
$B_1 + B_2$ does not contain a non-zero vector subspace.
\end{lemma}
\begin{proof}
We shall write $E_1 = E_{B_1}$ and $E_2 = E_{B_2}$. We let also $E_1 + E_2$ be the semi-normed space absorbed by $B_1 + B_2$ with the Minkowski functional associated to $B_1 + B_2$. We have the short exact sequence
\[ 0 \to E_1 \cap E_2 \to E_1 \oplus E_2 \to E_1 + E_2 \to 0, \]
where the first morphism maps $x \in E_1 \cap E_2$ to $(x, -x) \in E_1 \oplus E_2$ and the second morphism maps $(x, y) \in E_1 \oplus E_2$ to $x + y \in E_1 + E_2$. The space $E_1 + E_2$ is normed if and only if the kernel of the map $E_1 \oplus E_2 \to E_1 + E_2$ is closed and the kernel is closed if and only if its unit ball is a Banach disk. The unit ball of the kernel is the image of the unit ball of $E_1 \cap E_2$ by an injective map.
\end{proof}

\begin{lemma}(Grothendieck's lemma) 
Let $E$ be a vector space, let $B$ be a Banach disk in $E$ and $ \{ B_n \}_{n \in \N}$ be an increasing sequence of Banach disks of $E$ such that $B = \underset{n = 0}{\overset{\infty}\bigcup} B_n$. Then $B$ is absorbed by some $B_n$.
\end{lemma}
\begin{proof}
It is an immediate consequence of Lemma \ref{lemma:banach_disks} that the disk $B$ is compatible with all the $B_n$ for all $n \in \N$, since one has that  $B_n \subset B$. The space $E_B$ is a Banach space therefore it is a Baire space. According to the Baire's Theorem, for some $n \in \N$, $\ol{B_n}$ has a non-empty interior, where the closure is taken in $E_B$. As $B_n$ is absolutely convex, $\ol{B_n}$  contains a ball $D(\a)$ of radius $\a > 0$ in the Banach space $E_B$.
Let $x_0 \in D(\a) \subset \ol{B_n}$. We can find $y_0 \in B_n$ and $x_1 \in D(\a) \subset \ol{B_n}$ such that $x_0 = y_0 + \la x_1$, with $\la \in K^\times$ and $|\la| \le \frac{1}{2}$. Next
we choose $y_1 \in B_n$ and $x_2 \in D(\a) \subset \ol{B_n}$ such that $x_1 = y_1 + \la x_2$, etc. For each $j \in \N$, we see that 
\[ x_0 = \sum_{l = 0}^l \la^j y_l + \la^{ j-1} x_j. \]
The series $\underset{l = 0}{\overset{\infty}\sum} \la^j y_l$ converges in the Banach space $E_{B_n}$ and in that space, the norm of the sum is at most equal to $2$. So the sum  $\underset{l = 0}{\overset{\infty}\sum} \la^j y_l$ belongs to $\la^{-1} B_n$. In the Banach space $E_B$, $\la^{ -j} x_k \to 0$. Thus
$x_0 \in \la^{-1}B_n$ and $D(\a) \subset \la^{-1} B_n$.
\end{proof}

\begin{thm} (Buchwalter’s theorem)
Let $E$ be a complete bornological vector space whose bornology has a countable basis and $F$ be a complete bornological vector space. Let $f: E \to F$ be a surjective bounded linear map. Then $f$ is a strict epimorphism.
\end{thm}
\begin{proof}
	Let $\{ B_n \}_{n \in \N}$ be a basis of the bornology of $E$. We assume that the $B_n$ are Banach disks and that for all $n \in \N$, $B_n \subset B_{n+1}$. Let $C$ be a bounded Banach disk in $F$. Since $f$ is bounded, the subsets $f(B_n)$ are bounded Banach disks and $F = \underset{n = 0}{\overset{\infty}\bigcup} f(B_n)$. Lemma \ref{lemma:banach_disks} implies that for all $n \in \N$, the subset $f(B_n) \cap C$ is a Banach disk as $f(B_n) + C$, being bounded in $F$, does not contain any non-zero subspace. Moreover $C = \underset{n = 0}{\overset{\infty}\bigcup} (f(B_n ) \cap C)$. Then
Grothendieck’s lemma shows that $C$ is absorbed by one of the sets $f(B_n)$. It follows that there exist $\la \in K^\times$ and $n \in \N$ such that $C \subset \la f(B_n )$, which yields $C \subset f(\la B_n )$. Thus, we showed that the map $f$ is a strict epimorphism.
\end{proof}

\subsection*{De Wilde's Theorem}

Before proving our generalization of the de Wilde's Theorem, let's see how Theorem \ref{thm:web} generalizes Theorem \ref{thm:net}.

\begin{prop} \label{prop:net_web}
	Let $(E, \cB)$ be a separated convex bornological vector space and $\cN$ a
net on $E$ which is compatible with $\cB$. Then, for every $b : \N^\N \to |K^\times|^\N$ satisfying $(1)$ of Definition \ref{defn:net_compatible}, the couple  $(\cN, b)$ is a bornological web on $E$ such that 
\[ \cB_{(\cN,b)} \subset \cB \]
\end{prop}
\begin{proof}
	The first three conditions of the definition of bornological web are direct consequences of the definition of net. The last one is imposed by hypothesis and by condition $(1)$ of Definition \ref{defn:net_compatible}. So, we need only to check that $\cB_{(\cN,b)} \subset \cB$. And this follows directly from condition $(2)$ of  Definition \ref{defn:net_compatible}.
\end{proof}

\begin{cor}
	Theorem \ref{thm:net} is consequence of \ref{thm:web}.
\end{cor}
\begin{proof}
Direct consequence of Proposition \ref{prop:net_web}.
\end{proof}

Then, we need to introduce the topological version of the notion of web.

\begin{defn} \label{defn:top_webs}
Let $E$ be a Hausdorff locally convex space.
A map $\cW: \underset{k \in \N}\bigcup \N^k \to \sP(E)$ is called a \emph{topological web} if 
\begin{enumerate}
\item the image of $\cW$ consists of absolutely convex sets;
\item $\cW(\void) = E$;
\item Given a finite sequence $(n_0, \ldots, n_k)$, then $\cW(n_0, \ldots , n_k)$ is absorbed by
\[ \bigcup_{ n \in \N} \cW(n_0, \ldots, n_k, n). \]
\item  for every finite sequence $(n_0, \ldots , n_k, n_{k+1})$ one has
\[ \cW(n_0, \ldots , n_k, n_{k+1}) + \cW(n_0, \ldots , n_k, n_{k+1}) \subset \cW(n_0, \ldots , n_k). \]
\end{enumerate}
A separated locally convex space $E$ that carries a topological web is called
\emph{webbed locally convex space}.
Moreover, we say that $\cW$ is \emph{completing} if the following condition is satisfied: For every $s : \N \to \N$ and for every choice of $y_k \in \cW(s(1), \ldots, s(k))$ the series
\[ \sum_{k \in \N} y_k \]
converges topologically in $E$.
\end{defn}

\begin{rmk} \label{rmk:la}
Notice that condition (4) of the last definition when $K$ is non-Archimedean reduces to
\[ \cW(n_0, \ldots , n_k, n_{k+1}) \subset \cW(n_0, \ldots , n_k) \]
because in this case $\cW(n_0, \ldots , n_k, n_{k+1}) + \cW(n_0, \ldots , n_k, n_{k+1}) = \cW(n_0, \ldots , n_k, n_{k+1})$, and reduces to the condition
\[ 2 \cW(n_0, \ldots , n_k, n_{k+1}) \subset \cW(n_0, \ldots , n_k) \]
when $K$ is Archimedean. Hence, from here on we define the constant 
\[ \la = \begin{cases}  2 & \text{ if $K$ is Archimedean } \\ 1 & \text{ if $K$ is non-Archimedean }
\end{cases}. \]
\end{rmk}

Also for topological webs we will use the notation introduced in previous sections:

\[ \cW_{s,k} \doteq \cW(s(0), \ldots , s(k)), \text{ where } s: \N \to \N. \]

We need some technical lemmata. The following is the generalization of Proposition 5.2.1 of \cite{J}.

\begin{lemma} \label{lemma:top_equiv}
Let $E$ be a locally convex space and $\cW$ a topological web on $E$, then $\cW$ is completing if and only if for any open $0$-neighborhood $U \subset E$ and any $s: \N \to \N$ there exists a $k \in \N$ such that $\cW_{s, k} \subset U$.
\end{lemma}
\begin{proof}
Suppose that $\cW$ is completing. Then every sequence of elements $y_k \in \cW_{s, k}$ must be a zero sequence. Suppose that there exists a $0$-neighborhood such that for every $k \in \N$ one has $\cW_{s, k} \not\subset U$. In this way we can construct a sequence $y_k \in \cW_{s, k} - U$ which cannot be a zero sequence, thus $\cW$ cannot be completing. 

For the reverse implication, consider $s: \N \to \N$ and $y_k \in \cW_{s, k}$ for each $k \in \N$. So, for each $0$-neighborhood $U \subset E$ we can find a $k_0 \in \N$ such that $\cW_{s, k_0} \subset U$, therefore $\cW_{s, k} \subset U$ for all $k \ge k_0$. Applying inductively (4) of Definition \ref{defn:top_webs} we get that for any $m, k \in \N$, with $k \ge k_0$
\[ \sum_{n = 1}^p y_{k + n} \in \cW_{s, k + 1} + \ldots + \cW_{s, k + p} \subset \cW_{s, k} \subset U. \]
This shows that the sequence of partial sums $\underset{n = 1}{\overset{p}\sum} y_{k + n}$ for $p \to \infty$, is a zero sequence for the topology of $E$.
\end{proof}

\begin{lemma} \label{lemma:new_web}
Let $E$ be a Hausdorff locally convex space which is endowed with
a topological web $\cW$. Then the map $\cV: \bigcup_{k \in \N} \N^k \to \sP(E)$, defined by
\[ \cV (n_0 , \ldots , n_k ) := \frac{1}{\la^k} \cW (n_0 , \ldots , n_k ), \]
is again a topological web on $E$, where $\la \in K$ is as in Remark \ref{rmk:la}. Moreover, if $\cW$ is completing then also $\cV$ is.
\end{lemma}
\begin{proof}
$\cV$ clearly satisfies the first three conditions of Definition \ref{defn:top_webs}, so let's check the fourth.
\[ \la V (n_0 , \ldots , n_k , n_{k+1}) = \frac{1}{\la^{k -1}} \cW (n_0 , \ldots , n_k , n_{k+1} ) \subset \frac{1}{\la^k}W (n_0 , \ldots , n_k ) = \cV (n_0 , \ldots , n_k ). \]
Finally, let's suppose that $\cW$ is completing. Then, since the sets $\cW(n_0, \ldots , n_k)$ are absolutely convex they are in particular balanced, so
\[
 \cV (n_0, \ldots , n_k) = \frac{1}{\la^k} \cW(n_0, \ldots , n_k) \subset \cW(n_0, \ldots , n_k), \]
therefore the completing condition for $\cW$ implies the completing condition for $\cV$.
\end{proof}

From now on we will consider only completing topological webs.

\begin{defn} \label{defn:polar}
Let $E$ be a $K$-vector space. A semi-norm $p$ on $E$ is called \emph{polar} if 
\[ p = \sup \{ |f|  | f \in E^*, |f| \le p \} \]
where $E^*$ denotes the algebraic dual of $E$. A locally convex space $E$ is said to be polar if its topology can be defined by polar semi-norms.
\end{defn}

\begin{rmk}
If $K$ is spherically complete (hence also for $\R$ and $\C$), all locally convex spaces are polar, cf. Theorem 4.4.3 of \cite{PGS}.
\end{rmk}

\begin{notation}
Let $E$ be a locally convex space and $X \subset E$ any subset we define
\[ X^\circ \doteq \{ f \in E' | |f(x)| \le 1, \forall x \in E \} \]
\[ X^{\circ \circ} \doteq \{ x \in E | |f(x)| \le 1, \forall f \in X^\circ \} \]
where $E'$ is the continuous dual of $E$.
\end{notation}

In next lemma we use the notation $S_K = l^1_K$ if $K$ is Archimedean and $S_K = c_K^0$ if $K$ is non-Archimedean, where $l^1_K$ is the $K$-Banach space of summable sequences and $c_K^0$ is the $K$-Banach space of zero sequences.

\begin{lemma} \label{lemma:polar}
Let $E$ be a polar locally convex space over a spherically complete field $K$. Let $\{y_k\}_{k \in \N}$ be a sequence of elements of $E$ such that $\underset{k \in \N}\sum \mu_k y_k$ converges for each possible choice of $\mu_k \in K^{\circ}$. Then, $B = (\{y_k\}_{k \in \N})^{\circ \circ}$ is a Banach disk in $E$. 
\end{lemma}
\begin{proof}
By hypothesis the series $\underset{k \in \N}\sum \mu_k y_k$ is convergent when we choose $\{ \mu_k\}_{k \in \N} \in S_K$. Let's denote by $D \subset S_K$ the unit ball and with $e_k \in S_K$, for each $k \in \N$, the elements of the canonical Schauder base.

The map $T: S_K \to E$ defined $\{ \mu_k\} \mapsto \underset{k \in \N}\sum \mu_k y_k$ sends $D \subset S_K$ into $B = (\{y_k\}_{k \in \N})^{\circ \circ}$. Moreover, $T$ is adjoint to the map
\[ E' \to c_K^0: u \mapsto \{ \lt u , y_k \gt\}_{k \in \N} \]
so $T$ is $(\sigma(S_K, c_K^0), \sigma(E, E'))$-continuous, where $\sigma$ stands for the usual notation for weak topologies. The fact that $T(D)$ is $\sigma(E, E')$-bounded implies that $T(D)$ is $\sigma(E, E')$-compact for $K$ Archimedean (this follows from the Bourbaki-Alaoglu Theorem) and $\sigma(E, E')$-c-compact for $K$ non-Archimedean, because we are supposing $K$ spherically complete (for $K$ non-Archimedean and spherically complete we can apply Theorem 5.4.2 and 6.1.13 of \cite{PGS} to deduce weak-c-compactness from weak-boundedness). Hence $T(D)$ is weakly closed also for $K$ non-Archimedean because we can apply Theorem 6.1.2 (iii) of \cite{PGS}. Thus, since $T e_k = y_k$ for each $k$, and $T(D)$ is absolutely convex the Bipolar Theorem implies $T(D) = B$ (see Theorem 5.2.7 for the non-Archimedean version of the Bipolar Theorem). Moreover, $B$ is a bounded (because $E$ is polar and we can apply Theorem 5.4.5 of \cite{PGS}) Banach disk.
\end{proof}

So, from now on until the end of this section $K$ will be supposed to be spherically complete.

\begin{lemma} \label{lemma:web_ultraborn}
If $(E, \cW)$ is a polar webbed locally convex space, then also $E_\uborn$ is
a webbed locally convex space.
\end{lemma}
\begin{proof}
We shall show that given a topological web $\cW$ on $E$, then the topological web $\cV$ associated to $\cW$ as in Lemma \ref{lemma:new_web} is a topological web for $E_\uborn$. The only non-trivial thing to check is that $\cV$ is completing for the topology of $E_\uborn$. So, we reproduce here the argument of Theorem 13.3.3 of \cite{J} adapting it for any base field.

Let's consider $s: \N \to \N$ and $x_k \in \cV_{s, k}$ for each $k \in \N$. Then, by condition (4) of Definition \ref{defn:born_web} we know that 
\[ \la \cV_{s, k + 1} \subset \cV_{s, k} \then \la^{k -1} \cV_{s, k + 1} \subset \cW_{s, k}. \]
Thus, for each $x_k$ we can write
\[ x_k = \frac{y_k}{\la^{k-1}}, \ \ \text{ with } y_k \in \cW_{s, k} \]
and since $\cW_{s, k}$ are balanced, the assumption that $\cW$ is completing implies that the series
$\underset{k \in \N}\sum \mu_k y_k$ converges for the topology of $E$ for each possible choice of $\mu_k \in K^\circ$. Thus we can apply Lemma \ref{lemma:polar} to deduce that $B = (\{y_k\}_{k \in \N})^{\circ \circ}$ is a bounded Banach disk of $E$. This implies that $\underset{k \in \N}\sum x_k$ converges in $E_B$ and so also in $E_\uborn$.
\end{proof}

\begin{lemma} \label{lemma:web_top_born}
Let $E$ be a polar locally convex space. If $E$ is webbed, then $(E, \cB_\Ban)$ is a webbed convex bornological space with a bornological web $(\cV, b)$ that may be chosen in such a way that $\cB_{(\cV,b)}$ is finer than the von Neumann bornology of $E_\uborn$.
\end{lemma}
\begin{proof}
Let $\cW$ be a topological web on $E$. By Lemma \ref{lemma:new_web}  
\[ \cV (n_0, \ldots , n_k) \doteq \frac{1}{\la^k} \cW(n_0, \ldots , n_k) \]
is another topological web of $E$.
For $s: \N \to \N$ define $b(s): \N \to |K^\times|$ to be constant with value $1$. We claim that $(\cV, b)$ is a bornological web for $(E, \cB_\Ban)$. The first three conditions of bornological web are clear, so only the last one need to be checked. Given a sequence of elements $x_k \in \cV_{s, k}$ then we can define 
\[ x_k = \frac{y_k}{\la^{k-1}}, \ \ \text{ with } y_k \in \cW_{s, k} \]
and apply Lemma \ref{lemma:polar} to the sequence $\{y_k\}_{k \in \N}$ for obtaining that $\underset{k \in \N}\sum x_k$ converges in $E_B$, where $B = (\{y_k\}_{k \in \N})^{\circ \circ}$. So $\{y_k\}_{k \in \N}$ is a zero sequence in $E_B$ which implies that $\underset{k \in \N}\sum x_k$ converges for the topology of $E_\uborn$.

In order to prove that the $(\cV, b)$ is finer than the von Neumann bornoloy of $E_\uborn$, first let's notice that $(\cV, b)$ is also a bornological web for $(E_\uborn)^b$, since the bounded Banach disks that generate $\cB_\Ban$ are all bounded subsets for $E_\uborn$. Next, let $\{\mu_k\}_{k \in \N}$ be a sequence with values in $|K^\times|$. For every choice of $x_k \in \cV_{s,k}$, $k \in \N$, the value of
\[ \sum_{k > n + 1}^\infty x_k \]
belongs to $\ol{\cV_{s,n}}$. Hence $\tilde{\cV}_{s,n} \subset \ol{\cV_{s,n}}$, but $\underset{k \in \N}\bigcap \mu_k \ol{\cV_{s,n}} $ is bounded, since for any given closed and absolutely convex $0$-neighbourhood $U$ there is an index $n \in \N$ such that $\cV_{s,n} \subset U$, by Lemma \ref{lemma:top_equiv}. Hence, $\mu_n \ol{\cV_{s,n}}$, and consequently $\underset{k \in \N}\bigcap \mu_k \ol{V_{s,k}}$, is absorbed by $U$. Thus, we proved that $\cB_{(\cV,b)}$ is finer with respect to the canonical bornology of $E_\uborn$.
\end{proof}

\begin{rmk}
Since $\la = 1$ for $K$ non-Archimedean, last lemma proves that a topological web for a locally convex space $E$ is automatically a topological web for $E_\uborn$ in this case, and this essentially follows from the fact that for non-Archimedean base fields a sequence is summable if and only if it is a zero sequence.
\end{rmk}

\begin{lemma} \label{lemma:compatible_neumann}
Let $E$ be a polar webbed locally convex space. Then $(E_\uborn, \cB_\Ban)$
carries a bornological web $(\cV, b)$ such that the corresponding convex bornology $\cB_{(\cV,b)}$ is contained in the von Neumann bornology of $E_\uborn$.
\end{lemma}
\begin{proof}
By Lemma \ref{lemma:web_ultraborn} the ultrabornologification $E_\uborn$ of a polar webbed locally convex space $E$ is webbed. Then, Lemma \ref{lemma:web_top_born} applied to $E$ yields the assertion.
\end{proof}
Finally we are able to deduce the main result of this sectiom. Here we propose again the full statement of the theorem as already stated in the introduction.

\begin{thm}  
If $E$ is an ultrabornological locally convex space and
$F$ is a polar webbed locally convex space defined over a sphericaly complete field $K$, then every linear map $f: E \to F$ which has bornologically closed graph with respect to $(E, \cB_\Ban)$ and $(F, \cB_\Ban)$, is continuous even if regarded as a map $f: E \to F_\uborn$.
\end{thm}
\begin{proof}
We consider $(E, \cB_\Ban)$ as domain space of $f$ and $(F_\uborn, \cB_\Ban)$ as the codomain of $f$. Notice that the family of bounded Banach disks of $F$ and of $F_\uborn$ coincide and that, by Lemma \ref{lemma:web_ultraborn}, $F_\uborn$ carries a bornological web $(\cV, b)$ such that $\cB_{(\cV,b)}$ is finer than the canonical bornology of $F_\uborn$.
By hypothesis $f$ has bornologically closed graph with respect to $(E, \cB_\Ban)$ and $(F, \cB_\Ban)$, therefore we may apply Theorem \ref{thm:web} in order to see that $f$ is bounded, which also implies that $f$ is bounded if regarded as a map from $(E, \cB_\Ban)$ to $(F_\uborn)^b$.
Since $E$ is ultrabornological, we get that $f: E \to F_\uborn$ is continuous.
\end{proof}

\begin{rmk}
The proof of de Wilde's Theorem presented here closely follow the proof given in \cite{G}, adapting it in order to treat on the same footing both the Archimedean and the non-Archimedean base fields case. The main difference is the need of polarity assumption on $E$ and spherically completeness assumption on $K$ which are automatic in \cite{G}.
\end{rmk}

\subsection*{Applications to bornological algebras}

In this last section we show the applications which gave us the main motivations for writing down the proofs of the theorems discussed so far. The material of this section is mainly taken from one of the key technical point of author's Ph.D. thesis, \cite{Bam}.

\begin{defn} \label{defn:born_algebra}
A bornological $K$-vector space $A$ equipped with a bilinear associative function $A \times A \to A$, called \emph{multiplication map}, 
is said to be a \emph{bornological algebra} if the multiplication map is bounded.  We always suppose that $A$ has an identity and that the multiplication is commutative.
A morphism of bornological algebras is a bounded linear map that preserves multiplication and maps $1$ to $1$.
\end{defn}

Our next proposition is a generalization of Proposition 3.7.5/1 of \cite{BGR}, which holds for Banach algebras.

\begin{prop} \label{prop:algebra_morphism}
Let $A, B$ be bornological algebras over $K$ for which the underlying bornological vector space of $A$ is complete and the one of $B$ is a webbed bornological vector space and let $\phi: A \to B$ be an algebra morphism. Suppose that in 
$B$ there is a family of ideals $\cI$ such that
\ben
\item each $I \in \cI$ is bornologically closed in $B$ and each $\phi^{-1}(I)$ is bornologically closed in $A$;
\item for each $I \in \cI$ one has $\dim_K B/I < \infty$;
\item $\underset{I \in \cI}\bigcup I = (0)$.
\een
Then, $\phi$ is bounded.
\end{prop}
\begin{proof}
Let $I \in \cI$ and let's denote $\be: B \to B/I$ the qutient epimorphism and $\psi = \be \circ \phi$. Let $\ol{\psi}: A/\Ker(\psi) \to B/I$ denote the 
canonical injection, which give us the following commutative diagram
\[
\begin{tikzpicture}
\matrix(m)[matrix of math nodes,
row sep=2.6em, column sep=2.8em,
text height=1.5ex, text depth=0.25ex]
{ A            & B   \\
  A/\Ker(\psi) & B/I \\};
\path[->,font=\scriptsize]
(m-1-1) edge node[auto] {$\phi$} (m-1-2);
\path[->,font=\scriptsize]
(m-1-1) edge node[auto] {} (m-2-1);
\path[->,font=\scriptsize]
(m-1-2) edge node[auto] {$\be$} (m-2-2);
\path[->,font=\scriptsize]
(m-2-1) edge node[auto] {$\ol{\psi}$} (m-2-2);
\path[->,font=\scriptsize]
(m-1-1) edge node[auto] {$\psi$} (m-2-2);
\end{tikzpicture}.
\]
We have that $\Ker(\psi) = \phi^{-1}(I)$, therefore, since by hypothesis $B/I$ is finite dimensional, also $A/\Ker(\psi)$ is. Thus, both $B/I$ and $A/\Ker(\psi)$ are finite dimensional separated bornological algebras, when they are equipped with the quotient bornology. Therefore, their underlying bornological vector spaces are isomorphic to the direct product of a finite number of copies of $K$. So, $\ol{\psi}$ is bounded and this implies the boundedness of $\psi$.

Let's consider a sequence $\{a_n\}_{n \in \N} \subset A$ such that $\underset{n \to \infty}\lim a_n = 0$, bornologically. Then
\[ \be(\lim_{n \to \infty} \phi(a_n)) = \lim_{n \to \infty} (\be \circ \phi)(a_n) \]
since $\be$ is bounded, and therefore
\[ \lim_{n \to \infty} (\be \circ \phi)(a_n) = \lim_{n \to \infty} \psi(a_n) = \psi(\lim_{n \to \infty} a_n) = 0 \]
which implies that $\phi(\underset{n \to \infty}\lim a_n) \in I$. 
Since this must be true for any $I \in \cI$ and $\underset{I \in \cI}\bigcup I = (0)$ we deduce that $\phi(\underset{n \to \infty}\lim a_n) = 0$. This implies that the graph of $\phi$ is bornologically closed, because then for any sequence $\{ a_n \}_{n \in \N}$ in $A$ such that $\underset{n \to \infty}\lim a_n = a$ one has that
\[ \underset{n \to \infty}\lim ( a_n, \phi(a_n)) = (a, \phi(a)) \in \Gamma(\phi). \]
Now we can apply Theorem \ref{thm:web} to infer that $\phi$ is bounded.
\end{proof}

Let $\rho = (\rho_1, ..., \rho_n) \in \R_+^n$ be a polyradius. We denote by $W_K^n(\rho)$ the algebra of overconvergent (also called germs) analytic functions on the polycylinder of polyradius $\rho$. One can check that there is a bijection
\[ W_K^n(\rho) \cong \limind_{r > \rho} T_K^n(r) \]
where $T_K^n(r)$ denotes the algebra of strictly convergent analytic function on the polycylinder of polyradius $r$. Since $T_K^n(r)$ are $K$-Banach algebras and by the Identity Theorem for analytic functions the system morphism of the inductive system are monomorphism, then $W_K^n(\rho)$ has a canonical structure of complete bornological algebra. When $\rho = (1, \ldots, 1)$ we will simply write $W_K^n$. For a detailed discussion of the algebras $W_K^n(\rho)$, their properties and their relations with the classical affinoid algebras and the algebras of germ of analytic functions on compact Stein subsets of complex analytic spaces the reader can refer to chapter 3 of \cite{Bam}.

\begin{defn} \label{defn:dagger_algebra}
A \emph{strict $K$-dagger affinoid algebra} is a complete bornological algebra which is isomorphic to a quotient $W_K^n/I$, for an ideal $I \subset W_K^n$.

A (non-strict) \emph{$K$-dagger affinoid algebra} is a bornological algebra which is isomorphic to a quotient 
\[ \frac{W_K^n(\rho)}{I} \]
for an arbitrary polyradius $\rho$. 
\end{defn}

\begin{rmk}
It is easy to check that the underlying bornological vector space of a $K$-dagger affinoid algebra is an LB-space, hence in particular it is a webbed bornological vector space.
\end{rmk}

From Proposition \ref{prop:algebra_morphism} we can deduce the following result.

\begin{prop}
Every morphism between dagger affinoid algebras is bounded.
\end{prop}
\begin{proof}
	If $\phi: A \to B$ is an algebra morphism between strict $K$-dagger affinoid algebras then we can apply Proposition \ref{prop:algebra_morphism} choosing as family $\cI$ the family of all powers of maximal ideals of $B$. The only non-trivial fact to check for applying Proposition \ref{prop:algebra_morphism} is the requirement that all the elements of $\cI$ must be bornologically closed in $B$ and their preimages bornologically closed in $A$. But in Section 3.2 of \cite{Bam} is proved that all ideals of dagger affinoid algebras are bornologically closed, hence it follows that $\phi$ is bounded.

The non-strict case can be reduced to the strict case noticing that any non-strict dagger affinoid algebra can be written as a direct limit of strict ones and that every algebra morphism can be written as a morphism of direct systems of algebras, as explained in Section 3.2 of \cite{Bam}. Therefore, every morphism between non-strict dagger affinoid algebras can be written as a direct limit of bounded ones, hence it is bounded.
\end{proof}

We conclude this overview of applications of the bornological Closed Graph Theorem by saying that the last proposition can be generalized to encompass a more general class of bornological algebras used in analytic geometry: Stein algebras and (at least a big subclass of) quasi-Stein algebras, both dagger and non-dagger. The arguments for showing the boundedness of algebra morphisms for this class of bornological algebras become more involved and do not fit in this discussion. The reader can refer to \cite{BaBeKr} for such a study.

\end{document}